\documentclass{amsart}%
\usepackage{amsfonts}
\usepackage{amsmath}
\usepackage{amsthm,amstext,amsfonts,bm,amssymb}
\usepackage[below]{placeins}
\usepackage{amssymb,enumerate,hyperref}
\usepackage{xfrac}
\usepackage{tikz-cd}
\usepackage{refcount}
\usepackage[margin=1cm,singlelinecheck=off]{caption}
\usepackage{enumitem,mathtools}

\providecommand{\U}[1]{\protect\rule{.1in}{.1in}}
\providecommand{\comment}[1]{}

\newcommand{\CE}{\mathcal E}		

		\newcommand{\CN}{\mathcal N}

\newcommand{\CS}{\mathcal S}

\newcommand{\BN}{\mathbb N}
\newtheorem{theorem}{Theorem}[section]

\newtheorem*{namedtheorem}{\theoremname}
\newcommand{\theoremname}{testing}

\makeatletter
\newtheorem*{rep@theorem}{\rep@title}
\newcommand{\newreptheorem}[2]{%
\newenvironment{rep#1}[1]{%
 \def\rep@title{#2 \ref{##1}}%
 \begin{rep@theorem}}%
 {\end{rep@theorem}}}
\makeatother

\newreptheorem{theorem}{Theorem}
\newreptheorem{lemma}{Lemma}
\newreptheorem{proposition}{Proposition}
\newtheorem{claim}[theorem]{Claim}

\newtheorem{corollary}[theorem]{Corollary}

\newtheorem{lemma}[theorem]{Lemma}

\newtheorem*{lemma*}{Lemma}

\theoremstyle{definition}

\newtheorem*{definition*}{Definition}

\newcommand{\inj}{\mathrm{inj}}

\newcommand{\BZ}{\mathbb{Z}}

\newcommand{\BR}{\mathbb{R}}

\newcommand{\BH}{\mathbb{H}}
\newcommand{\rank}{\mathrm{rank}}

\title{Subgroups of bounded rank in hyperbolic $3$-manifold groups}
\author{Ian Biringer}

\begin{document}
\maketitle
\begin{abstract}
We prove a finiteness theorem for subgroups of bounded rank in  hyperbolic $3$-manifold groups. As a consequence, we show that every bounded rank covering tower of closed hyperbolic $3$-manifolds  is a tower of finite covers associated to a fibration over a $1$-orbifold. \end{abstract}

\section{Introduction}

Suppose that $M$ is an orientable $3$-manifold and $O$ is a $1$-dimensional orbifold, both without boundary, so that $O$ is homeomorphic to one of  
\begin{equation} \BR, \  \BR/(x\mapsto -x), \  S^1, \text{ or } S^1 / (z\mapsto -z). \label{whatiso}\tag{$\diamond$}\end{equation}
A \emph{fibration of $M$ over $O$} is a fiber orbibundle $f : M \longrightarrow O$, e.g.\ as defined in \S 3 of \cite{biringer2021chabauty}. If $p\in O$, the preimage $S=f^{-1}(p) \subset M$ is called a \emph{regular fiber} or a \emph{singular fiber} depending on whether $p$ is a regular or singular point of $O$. Any singular fiber is a one-sided non-orientable surface embedded in $M$, while a regular fiber is a two-sided orientable surface. The preimage of any closed interval contained in the regular part of $O$ is  an embedded trivial interval bundle $S \times [0,1] \hookrightarrow M$, while the preimage of any closed interval containing a single singular point is a twisted interval bundle over the corresponding singular fiber. So, depending on the homeomorphism type of $O$ as in  \eqref{whatiso}, the manifold $M$ is homeomorphic to either: 
a trivial interval bundle $S \times \BR$ over an orientable surface $S$, a twisted open interval bundle over a nonorientable surface,  a mapping torus over an orientable surface, or a manifold obtained by gluing two copies of a twisted interval bundle over a non-orientable surface together along their boundaries.

A fibration of $M$ over $O$ gives a short exact sequence
\begin{equation}\label{SES}\tag{$\heartsuit$}
1 \longrightarrow \pi_1 S \longrightarrow \pi_1 M \longrightarrow Q \longrightarrow 1,
\end{equation}
where $S$ is a regular fiber, and the quotient $Q$ is either trivial, $\BZ/2\BZ$, infinite cyclic or infinite dihedral, depending on $O$ as in \eqref{whatiso}.  We call the normal subgroup $\pi_1 S \leq \pi_1 M$ the \emph{fiber subgroup}; note that normality means that this subgroup is well-defined by the fibration, independent of basepoints. 

Our main theorem is as follows.
\begin{theorem}[Finiteness for Bounded Rank Subgroups]\label {main} Let $M$ be an orientable, hyperbolizable $3$-manifold with finitely generated fundamental group and no $\BZ^2$ subgroups in $\pi_1 M$. Given $k\in \BN$, there is a finite set $\mathcal S$ of subgroups of $\pi_1 M$, such that for any subgroup $H\leq \pi_1 M$ with rank at most $k$, we have either
\begin{enumerate}
	\item $H = H_1 \star \cdots \star H_n \star F$, where each $H_i$ is conjugate in $\pi_1 M$ to an element of our finite set of subgroups $\mathcal S$, and $F\leq G$ is free,
	\item the cover $N$ corresponding to $H$ is compact and fibers over a $1$-orbifold, with fiber subgroup an element of $\mathcal S$.
\end{enumerate}
\end{theorem}

The proof uses technology developed by the author and Souto in \cite{biringer2017thick}. Briefly, equip $M$ with a convex-cocompact hyperbolic metric and let $N$ be a cover of $M$ such that $\pi_1 N$ is freely indecomposable and has bounded rank. The main theorem of \cite{biringer2017thick} says that the convex core of $N$ decomposes as a union of \emph{product regions} and \emph{building blocks}. The product regions are homeomorphic to $S \times I$, in such a way that the level surfaces $S \times t$ have bounded geometry, although the  region itself can be very wide. The building blocks all have bounded diameter, and hence only finitely many topological types. We show that if $N$ has a sufficiently wide product region, then it fibers as in (2), while otherwise, the convex core of $N$ has bounded diameter, which implies $\pi_1 N$ is one of finitely many subgroups $H_i$.

The case of Theorem \ref{main} (1) where $H$ has infinite index (or equivalently, the fact that there are only finitely many conjugacy classes of one-ended infinite index subgroups $H \leq \pi_1 M$ of bounded rank) can be deduced from earlier arguments in Kapovich-Weidmann~\cite{kapovich2005kleinian}, specifically their Theorem 7.5. Very recently, Weidmann-Weller \cite{weidmann2024foldings} have extended Kapovich-Weidmann's work to apply to $M$ that have rank $2$ cusps. At heart, all these arguments are similar, in that Kapovich-Weidmann and Weidmann-Weller use an algebraic version of the `carrier graphs' that are the essential tool in part of Biringer-Souto \cite{biringer2017thick}. However, we do not know how to prove Theorem~\ref{main} in full without relying on the rest of  \cite{biringer2017thick}. Also, the proof of Theorem~\ref{main} via the machinery of \cite{biringer2017thick} is quick and geometrically transparent.


As an application, we can characterize infinite chains of bounded rank subgroups of closed hyperbolic $3$-manifold groups. If $M$ is a closed $3$-manifold that fibers over a $1$-orbifold, then we have the short exact sequence \eqref{SES}, with $Q=\BZ$ or $D_\infty$, and any chain of finite index subgroups in $Q$ pulls back to a chain of finite index subgroups of $G$ with uniformly bounded rank.  We show that this is the only way to produce infinite chains of bounded rank.

\begin{corollary}\label{chains}
Suppose that $M$ is a closed, orientable hyperbolic $3$-manifold, and $$\pi_1 M \geq H_1 > H_2 > \cdots$$ is a chain of finite index subgroups such that $\sup_i \rank H_i <\infty$. Then after passing to a subsequence, $H_1$ is fibered over $S^1$ or $S^1/(z\mapsto -z)$ with associated SES $$1 \longrightarrow \pi_1 S \longrightarrow H_1 \longrightarrow Q \longrightarrow 1,$$ and $(H_i)$ is a chain of finite index subgroups of $H_1$ that is the preimage of a chain of finite index subgroups of $Q$.
\end{corollary}
\begin{proof}
Say $M,H_i$ are as above. Since the subgroups $H_i$ are finite index in $M$, are $1$-ended, and hence do not split nontrivially as free products. As the index of $H_i$ increases with $i$, none of these groups are conjugate to each other, so Theorem \ref{main} says that there are only finitely many $H_i$ that do not fiber over a $1$-orbifold with fiber subgroup an element of $\CS$. Remove these finitely many $H_i$, and pass to a subsequence so that for each $i$, the fiber subgroup is some fixed $K\leq \pi_1 M$.
\end{proof}

For context, let $M$ be a closed $3$-manifold. Lackenby \cite{Lackenbyexpanders} defined the \emph{rank gradient} of a chain $\pi_1 M \geq H_1 \geq H_2 \geq \cdots$ of finite index subgroups to be the limit
$$RG(H_i) := \lim_{i\to \infty} \frac{ \rank(H_i)-1}{[\pi_1 M:H_i]}.$$
Multiple authors have studied chains where $RG(H_i)=0$, i.e.\ where the rank grows sublinearly in the index. For instance, DeBlois--Frield--Vidussi~\cite{deblois2014rank} have shown that if $\phi : \pi_1 M \longrightarrow \BZ$ is a homomorphism, the chain $H_i = \phi^{-1}(i\BZ)$ has zero rank gradient if and only if $\phi$ is the surjection in the short exact sequence \eqref{SES} associated to a fibration of $M$ over the circle. In general, if $M$ is hyperbolic then every known example of a chain of subgroups of $\pi_1 M$ with zero rank gradient is constructed by modifying a chain coming from a fibration over a $1$-orbifold. It is unknown whether there are qualitatively different examples. For instance, let $(H_i)$ be a chain of subgroups of $\pi_1 M$, let $N_i$ be the associated tower of covers of $M$, and let $h(N_i)$ be their Cheeger constants. Is it true that $RG(H_i)=0 $ implies $h(N_i)\to 0$? It is easy to see that the Cheeger constant goes to zero in a chain coming from a fibration. A positive answer to this question would also disprove Gaboriau's \emph{Fixed Price Conjecture} in measurable dynamics, via work of Abert-Nikolov \cite{Abertrank}.

\subsection{Organization}

\S \ref{background} contains background necessary for the proof. In particular, in \S \ref{hyp} we review some hyperbolic geometry and prove a lemma that allows us to recognize when an embedded incompressible surface represents a fiber in a fibration of $M$ over a $1$-orbifold, while in \S \ref{thickbdd} we review the main theorem of Biringer-Souto~\cite{biringer2017thick}. The proof of Theorem \ref{main} is presented in \S \ref{pfsec}.

\subsection{Acknowledgements}

The author was partially supported by NSF CAREER award 1654114. This paper began as part of a joint project with Edgar A.\ Bering IV and Nir Lazarovich. Thanks to them for the inspiration!

\section{Background}
\label{background}

\subsection{Hyperbolic geometry}\label{hyp}
A (complete) \emph{hyperbolic $3$-manifold} is a quotient $M=\Gamma \backslash \BH^3$, where $\Gamma$ acts freely and properly discontinuously by isometries. A topological $3$-manifold is called \emph{hyperbolizable} if it is homeomorphic to a hyperbolic $3$-manifold as above. If $M$ is hyperbolic, the \emph{convex core} of $M$ is the quotient $$CC(M):= \Gamma \backslash CH(\Lambda(\Gamma)),$$ where $\Lambda(\Gamma)\subset \partial \BH^3$ is the \emph{limit set} of $\Gamma$, and $CH( \cdot )$ denotes the hyperbolic convex hull. See e.g. \cite{Matsuzakihyperbolic,Benedettilectures} for details. Since $CC(M)$ may be less than $3$-dimensional, it is sometimes more convenient to work with its closed $1$-neighborhood $CC_1(M)$. 

 If $M$ is an orientable hyperbolic $3$-manifold with finitely generated fundamental group, the Tameness Theorem of Agol \cite{Agoltameness} and Calegari-Gabai \cite{Calegarishrinkwrapping} says that $M$ is homeomorphic to the interior of a compact $3$-manifold with boundary. Equivalently, the ends of $M$ are all `tame', i.e.\ they all have neighborhoods homeomorphic  to $S\times (0,\infty)$, where $S$ is a closed orientable surface. Assuming for simplicity that $M$ has no cusps, work of Thurston, Bonahon and Canary \cite{Bonahonbouts,Canaryends} implies that each end $\CE$ of $M $ is  either \emph{convex cocompact}, meaning that it has a neighborhood that lies outside $CC(M)$, or \emph{degenerate}, meaning that it has a neighborhood $U \cong S\times (0,\infty)$ that contains a sequence of bounded area surfaces $f_i : S \longrightarrow U$ in the homotopy class of a level surface, where $f_i$ exits the end as $i\to \infty$. See e.g.\ \cite{Matsuzakihyperbolic} for more details and precise definitions. As an example, if $M$ fibers over the circle with fibers homeomorphic to a surface $S$, then the obvious infinite cyclic cover $N$ of $M$ is homeomorphic to $S\times \BR$ and both ends of $N$ are degenerate: one can take the required bounded area surfaces to be all the lifts of a fixed surface $S \longrightarrow M$ in the homotopy class of the fiber.

     Suppose that $M$ is an orientable hyperbolic $3$-manifold and $S \looparrowright M$ is an immersed $\pi_1$-injective closed surface. The cover $M_S $ of $M$ corresponding to $\pi_1 S$ is homeomorphic to $S\times \BR$; this follows from the Tameness Theorem and some standard $3$-manifold topology arguments, c.f.\ \cite{Hempel3-manifolds}. We call  $M_S$ \emph{doubly degenerate} if $M_S$ has two degenerate ends; in this case we also call $S$ `doubly degenerate'.

 \begin{lemma}\label{covering Karen}
Suppose that $M$ is an orientable hyperbolic $3$-manifold and $S\hookrightarrow M$ is a doubly degenerate, \underline{embedded} $\pi_1$-injective closed orientable surface. Then $S$ is a regular fiber in a fibration of $M$ over a $1$-orbifold.
 \end{lemma}
 \begin{proof}
Let $M_S $ be the cover of $M$ corresponding to $S$, so $M_S$ is a doubly degenerate hyperbolic $3$-manifold homeomorphic to $S\times \BR$. Thurston's covering theorem \cite[Theorem 9.2.2]{Thurstongeometry} says that either $\pi: M_S \longrightarrow M$ is finite-to-one, or $\pi$ factors as
\begin{equation}\label{factor}M_S \longrightarrow N \overset{\rho}{\longrightarrow} M,\end{equation}
where the first map is a cyclic covering map onto a closed hyperbolic $3$-manifold fibering over the circle and the second map is a finite cover. 

Suppose first that $\pi: M_S\longrightarrow M$ is finite-to-one. Then $\pi_1 M$ has a finite index surface subgroup, so it is finitely generated and does not split as a free product. Hence $M$ is tame and any standard compact core $C\subset M$ has incompressible boundary. Here, a `standard compact core' $C$ is a compact submanifold such that $M\setminus int(C)$ is homeomorphic to $\partial C\times [0,\infty)$. Each component $T \subset \pi^{-1}(\partial C)$ is then an incompressible, embedded closed surface in $M_S \cong S \times \BR$, and therefore is a `level surface', isotopic to $S \times \{t\}$. The restriction $\pi |_T$ is an embedding: if not, it would nontrivially cover a component of $\partial C$, but since $\pi |_T$ is homotopic in $M$ to the embedded surface $S$, work of Freedman-Hass-Scott (c.f.\ Lemma 3.1 in \cite{biringer2017thick}) implies that $\pi |_{T}$ is homotopic to an embedding within an arbitrarily small neighborhood of its image, which is impossible since components of $\partial C$ are two-sided, and hence their regular neighborhoods are products. Since $\pi^{-1}(C)$ is connected and is bounded by level surfaces, it must be a trivial interval bundle bounded by $\pi^{-1}(\partial C)$, which has two components. Since each of these components embeds under $\pi$, the covering map $\pi : \pi^{-1}(C) \longrightarrow C$ is either a homeomorphism or is 2-1, depending on whether the two components of $\pi^{-1}(\partial C)$ have distinct $\pi$-images or not.  In the first case, $C$ is a trivial interval bundle, and hence $M$ is an open trivial interval bundle with $S$ a fiber, while in the second case Proposition 4.1 in \cite{Waldhausenirreducible} implies that $M$ is an open twisted interval bundle with $S$ a regular fiber. So, we're done.

Now suppose that $\pi$ factors as in \eqref{factor}. Every component of $\rho^{-1}(S) \subset N$ is an incompressible embedded surface that is homotopic to a finite cover of the fiber of $N$, so every component is actually in the homotopy class of the fiber. Hence $\rho^{-1}(S)$ cuts $N$ into a collection of trivial interval bundles $S \times [0,1]$, and each component of $\rho^{-1}(S)$ projects homeomorphically into $M$. As in the previous paragraph, this implies that $S$ cuts $M$ into pieces that are either trivial interval bundles $S \times [0,1]$ or twisted interval bundles with $S$ a regular fiber. So, $M$ fibers.
 \end{proof}
 
 \subsection{Thick manifolds with bounded rank}\label{thickbdd}
In this subsection we review some material from \cite{biringer2017thick}.
 Let $M$ be an orientable hyperbolic $3$-manifold. A \emph{product region} in $M$ is the image $U \subset M$  of a proper embedding $$\Sigma_g \times I \longrightarrow M, \ \ I=[0,1],[0,\infty), \text{or } (-\infty,\infty),$$  such that  for some regular neighborhood $\CN(U) \supset U$, we have:
	\begin{enumerate}
		\item  every point $p\in U$  is in the image of a NAT simplicial ruled surface $\Sigma_g \longrightarrow \CN(U)$  that is a homotopy equivalence,
		\item each component $ S \subset \partial U$ lies in the $1$-neighborhood of  another such NAT simplicial ruled surface.
	\end{enumerate}
Here, a simplicial ruled surface (SRS) is a map from a triangulated surface $\Sigma$, where edges map to geodesics and where the image of each triangle is foliated by geodesics. Equipping $\Sigma_g$ with the pullback metric, we say that the SRS is NAT (or `not accidentally thin') if there is no simple closed curves on $\Sigma_g$ with length less than $\epsilon$ that is nullhomotopic in $M$. The reader can see \cite[\S 5]{biringer2017thick} for precise definitions, but the main point is that NAT SRSs  $\Sigma_g \longrightarrow M$ 
in $\epsilon$-thick\footnote{The \emph{injectivity radius} of a hyperbolic $3$-manifold $M$, written $\inj(M)$, is the half the length of the shortest homotopically essential loop in $M$, and $M$ is called \emph{$\epsilon$-thick} if $\inj(M)\geq \epsilon$.} hyperbolic $3$-manifolds $M$ satisfy a `bounded diameter lemma', i.e.\ the diameter of $\Sigma_g$ in  the pullback metric is bounded above by a constant depending only on $g,\epsilon$. See \S 5.5 of \cite{biringer2017thick}. For the purposes of this paper, it's sufficient to think of a product region as just a submanifold homeomorphic to a surface cross an interval, where the geometry is bounded in the surface direction. The phrasing with SRSs is just a useful way to formulate this without specifying a priori what `bounded' means.

 In \cite[Theorem 13.1]{biringer2017thick}, the authors prove the following geometric decomposition theorem for convex cores of $\epsilon$-thick hyperbolic $3$-manifolds with bounded rank.

\begin{theorem}\label{BS}
Fix $k\in \BN$ and some sufficiently small $\epsilon>0$. Then there are constants $n=n(k),$ $ g=g(k)$, $B=B(k,\epsilon)$ as follows.

Suppose $M$ is a  complete,  orientable hyperbolic $3$-manifold with $$\rank(\pi_1(M))\le k, \ \ \inj(M)\ge\epsilon,$$ and assume that $\pi_1 M$ is freely indecomposable. Then $CC(M)$ contains a collection $\mathcal U$ of at most $n$ product regions, each with genus at most $g$, such that every component of $CC_1(M) \setminus (\cup_{U \in \mathcal U} int(U) ) $ has diameter at most $B$.
\end{theorem}
 
 The statement above is slightly different from that given in \cite[Theorem 13.1]{biringer2017thick}. First, the `freely indecomposable' assumption in the statement above implies that there are no essential simple closed curves on $\partial CC(M)$ that are compressible in $M$; this is a  stronger version of a related hypothesis in the statement in \cite{biringer2017thick}. Also, in \cite{biringer2017thick} the authors work with the convex core and its interior rather than the $1$-neighborhood and the convex core itself, under a standing assumption that the convex core is $3$-dimensional, but the version above follows formally from theirs.

\medskip

\section{The proof of Theorem \ref{main}}\label{pfsec} Fix an orientable hyperbolic $3$-manifold $ M$ with finitely generated fundamental group and no rank two cusps. Here's what we will actually prove.

\begin{theorem}\label{realthm}
Given $k\in \BN$, there is some $g=g(k)$ such that there are only finitely many isomorphism types of subgroups $H \leq \pi_1 M$ such that 
\begin{enumerate}
    \item[(a)] $H$ is freely indecomposable and $\rank(H)\leq k$,
    \item[(b)] it's not the case that $M$ is compact, $H \leq \pi_1 M$ has finite index, and the associated cover $N$ of $M$ fibers over a compact $1$-orbifold with regular fiber of genus at most $g$.
\end{enumerate}
\end{theorem}

Let's show how to derive Theorem \ref{main} from Theorem \ref{realthm}. 

\begin{proof}[Proof of Theorem \ref{main}]
    Let $\mathcal S^1$ be a minimal set of subgroups of $G$ representing all the conjugacy classes of one-ended subgroups $H \leq G$ with rank at most $k$ that satisfy (b) above.  When $M$ is noncompact, we set $\mathcal S=\mathcal S^1$. Otherwise, we set $\mathcal S$ to be the union of $\mathcal S^1$ with all doubly degenerate closed surface subgroups of $G$ that have genus at most $g$.

Since $G$ is a hyperbolic group, for any fixed finitely presented one-ended hyperbolic group $H$, there are a finite number of conjugacy classes of subgroups of $G$ isomorphic to $H$, see Delzant \cite{delzant1995image}. Any one-ended group is freely indecomposable, so by this and Theorem \ref{realthm}, the set $\mathcal S^1$ is finite. Delzant's theorem also implies that up to conjugacy, there are only finitely many doubly degenerate closed surface subgroups $K\leq G$ with genus at most $g$. However, if $M$ is compact, Thurston's covering theorem \cite[Theorem 9.2.2]{Thurstongeometry} implies that any such $K\leq G$ is a normal subgroup of a finite index subgroup $H_K \leq G$. Since there only finitely many conjugacy classes of such $K$, we can take the indices $[G : H_K]$ to be bounded, and hence finitely many $H_K$ suffice, implying that there are only finitely many $K \leq G$, even without identifying conjugates. So, $\mathcal S $ is always finite.

Now suppose $H \leq G$ has rank at most $k$. Since $G$ is torsion free, it follows from Grushko's Theorem and Stallings' Theorem, c.f.\ \cite{scott1979topological}, that $H$ can be written as $$H = H_1 \star \cdots \star H_n \star F,$$
where the $H_i$ are one-ended, $F$ is a free group, and all these free factors have rank at most $k$. If all the $H_i$ satisfy (b) above, they are conjugate into $\mathcal S$ and we're done. Otherwise, some $H_i$ has finite index, so the free product decomposition must be trivial, i.e.\ $H=H_i$, and the cover $N\longrightarrow M$ corresponding to $H$ fibers over a compact $1$-orbifold with regular fiber a surface of genus at most $g$, and the fiber subgroup of $H$ lies in $\mathcal S$ by construction.
\end{proof}

The rest of the section is devoted to the proof of Theorem \ref{realthm}. Since the conclusion of the theorem is topological, we may assume that $M$ is convex co-compact, say by Thurston's Haken hyperbolization theorem \cite{Kapovichhyperbolic} in the noncompact case.

\begin{claim}\label{shortprods}
There is some $L=L(g,M)$ as follows. Suppose that $N \longrightarrow M$ is a locally isometric covering map and $U \subset N$ is a product region with genus at most $g$. Then either $\mathrm{width}(U) \leq L$, or there is a fibration of $N$ over a $1$-dimensional orbifold where $U$ is a collar neighborhood of a regular fiber.
\end{claim}

In the statement of the claim, if $U$ is a compact product region, $\mathrm{width}(U)$ is defined to be the infimal length of a path in $U$ between the two boundary components of $U$. If $U$ is noncompact, we set $\mathrm{width}(U) := \infty$.

\begin{proof}
	It suffices to prove the claim for product regions of fixed genus $g$. Hoping for a contradiction, let $\rho_i:N_i \longrightarrow M$ be a sequence of locally isometric covers containing product regions $U_i \subset N_i$ with genus $g$, where $\mathrm{width}(U_i) \to \infty$, and where no $U_i$ is a collar neighborhood of a regular fiber in a fibration of $N_i$.

 Pick base points $p_i \in U_i$ such that $d(p_i,\partial U_i)\to \infty$. By Lemma 6.20\footnote{In \cite{biringer2017thick}, results are often stated using `width relative to the $\epsilon$-thin part' instead of width as defined here. However, in the current setting we are only looking at covers of a fixed convex cocompact $M$, so if $\epsilon$ is chosen smaller than the injectivity radius of $M$ then width and relative width coincide.} in \cite{biringer2017thick}, we can assume after passing to a subsequence that the sequence $(N_i,p_i)$ converges geometrically to a pointed doubly degenerate hyperbolic $3$-manifold $(N_\infty,p_\infty)$, where $N_\infty \cong \Sigma_g \times \BR$. Moreover, if we choose level surfaces $S_i \subset U_i$ at bounded distance from $p_i$ (see Lemma 6.5 of \cite{biringer2017thick}), then if $(\phi_i)$ is a sequence of almost isometric maps witnessing the geometric convergence (see Definition 9.1 of \cite{biringer2017thick}), for large $i$ we have that the image of $\phi_i$ contains $S_i$ and $\phi_i^{-1}(S_i)$ is a level surface in $N_\infty$.
 
 By Arzela-Ascoli, after passing to a subsequence we can assume that the maps $\rho_i\circ \phi_i$ converge to a locally isometric covering map $\rho_\infty : N_\infty \longrightarrow M$. It follows that for large $i$, the groups $(\rho_i)_*(\pi_1 S_i) $ and $(\rho_\infty)_*(\pi_1 N_\infty)$ are conjugate in $\pi_1 M$. So, $S_i$ is a doubly degenerate incompressible embedded surface in $N_i$, and hence is a regular fiber in some fibration over a $1$-orbifold by Lemma \ref{covering Karen}, a contradiction.
\end{proof}

 So, let $\epsilon>0$ be smaller than the injectivity radius of $M$, and small enough so that Theorem \ref{BS} holds. Let $g=g(k)$ be as in Theorem \ref{BS}. Let $N$ be a cover of $M$ with $\rank (N) \leq k$, with $\pi_1 N$ freely indecomposable, and assume that it's not the case that $N$ is compact and fibers over a $1$-orbifold with regular fiber a genus at most $g$ surface. We will show that $N$ is homotopy equivalent to a simplicial complex with at most $V=V(k,M)$ simplices; it will follow that $\pi_1 N$ has one of finitely many isomorphism types.
 
First, it could be that $N$ is noncompact, but fibers over a $1$-orbifold with a regular fiber of genus at most $g$. In this case, $N$ is homotopy equivalent to a (possibly non-orientable) surface with bounded complexity, so we're done. We may then assume that $N$ does not fiber over a $1$-orbifold with regular fiber of genus at most $g$, in which case Claim \ref{shortprods} says that all genus at most $g$ product regions in $N$ have width at most some $L=L(k,M)$. Since the width of a thick product region bounds its diameter \cite[Fact 6.4]{biringer2017thick}, by Theorem \ref{BS} we have that $CC_1(N)$ has diameter at most some constant $D=D(k,M)$. Pick a maximal $\epsilon/2$-separated set $S\subset CC_1(N)$. The number of such points is at most linear is $\mathrm{vol} \, CC_1 (N)$, which is bounded above by a function of $D$. Let $\mathcal N$ be the nerve complex of the set of $\epsilon$-balls around points in $S$. Then $\mathcal N$ has bounded degree, so the number of simplices in $\mathcal N$ is bounded by some $V=V(k,M)$. By the nerve lemma (c.f.\ Corollary 4G.3 in \cite{Hatcheralgebraic}), $\mathcal N$ is homotopy equivalent to the union of all the associated $\epsilon$-balls. This union deformation retracts onto $CC_1(N)$, via the closest point retraction, and $N$ does as well, so $N$ and $\mathcal N$ are homotopy equivalent as desired.

\bibliographystyle{amsplain}
\bibliography{total}

\providecommand{\bysame}{\leavevmode\hbox to3em{\hrulefill}\thinspace}
\providecommand{\MR}{\relax\ifhmode\unskip\space\fi MR }
\providecommand{\MRhref}[2]{%
  \href{http://www.ams.org/mathscinet-getitem?mr=#1}{#2}
}
\providecommand{\href}[2]{#2}
\begin{thebibliography}{10}

\bibitem{Abertrank}
Mikl{\'o}s Ab{\'e}rt and Nikolay Nikolov, \emph{Rank gradient, cost of groups
  and the rank versus {H}eegaard genus problem}, J. Eur. Math. Soc. (JEMS)
  \textbf{14} (2012), no.~5, 1657--1677. \MR{2966663}

\bibitem{Agoltameness}
Ian Agol, \emph{Tameness of hyperbolic 3-manifolds}, arXiv:math.GT/0405568.

\bibitem{Benedettilectures}
Riccardo Benedetti and Carlo Petronio, \emph{Lectures on hyperbolic geometry},
  Universitext, Springer-Verlag, Berlin, 1992. \MR{1219310 (94e:57015)}

\bibitem{biringer2021chabauty}
Ian Biringer, Nir Lazarovich, and Arielle Leitner, \emph{On the {C}habauty
  space of {PSL}(2,{R}), i: lattices and grafting}, arXiv preprint
  arXiv:2110.14401 (2021).

\bibitem{biringer2017thick}
Ian Biringer and Juan Souto, \emph{Thick hyperbolic 3-manifolds with bounded
  rank}, arXiv preprint arXiv:1708.01774 (2017).

\bibitem{Bonahonbouts}
Francis Bonahon, \emph{Bouts des vari\'et\'es hyperboliques de dimension
  {$3$}}, Ann. of Math. (2) \textbf{124} (1986), no.~1, 71--158. \MR{MR847953
  (88c:57013)}

\bibitem{Calegarishrinkwrapping}
Danny Calegari and David Gabai, \emph{Shrinkwrapping and the taming of
  hyperbolic 3-manifolds}, J. Amer. Math. Soc. \textbf{19} (2006), no.~2,
  385--446 (electronic). \MR{MR2188131 (2006g:57030)}

\bibitem{Canaryends}
Richard~D. Canary, \emph{Ends of hyperbolic {$3$}-manifolds}, J. Amer. Math.
  Soc. \textbf{6} (1993), no.~1, 1--35. \MR{MR1166330 (93e:57019)}

\bibitem{deblois2014rank}
Jason DeBlois, Stefan Friedl, and Stefano Vidussi, \emph{Rank gradients of
  infinite cyclic covers of 3-manifolds}, Michigan Mathematical Journal
  \textbf{63} (2014), no.~1, 65--81.

\bibitem{delzant1995image}
Thomas Delzant, \emph{L'image d'un groupe dans un groupe hyperbolique},
  Commentarii Mathematici Helvetici \textbf{70} (1995), 267--284.

\bibitem{Hatcheralgebraic}
Allen Hatcher, \emph{Algebraic topology}, Cambridge University Press,
  Cambridge, 2002. \MR{MR1867354 (2002k:55001)}

\bibitem{Hempel3-manifolds}
John Hempel, \emph{3-manifolds as viewed from the curve complex}, Topology
  \textbf{40} (2001), no.~3, 631--657. \MR{MR1838999 (2002f:57044)}

\bibitem{kapovich2005kleinian}
Ilya Kapovich and Richard Weidmann, \emph{Kleinian groups and the rank
  problem}, Geometry \& Topology \textbf{9} (2005), no.~1, 375--402.

\bibitem{Kapovichhyperbolic}
Michael Kapovich, \emph{Hyperbolic manifolds and discrete groups}, Progress in
  Mathematics, vol. 183, Birkh\"auser Boston Inc., Boston, MA, 2001.
  \MR{1792613 (2002m:57018)}

\bibitem{Lackenbyexpanders}
Marc Lackenby, \emph{Expanders, rank and graphs of groups}, Israel J. Math.
  \textbf{146} (2005), 357--370. \MR{2151608 (2006c:20068)}

\bibitem{Matsuzakihyperbolic}
Katsuhiko Matsuzaki and Masahiko Taniguchi, \emph{Hyperbolic manifolds and
  {K}leinian groups}, Oxford Mathematical Monographs, The Clarendon Press
  Oxford University Press, New York, 1998, Oxford Science Publications.
  \MR{MR1638795 (99g:30055)}

\bibitem{scott1979topological}
Peter Scott and Terry Wall, \emph{Topological methods in group theory},
  Homological group theory (Proc. Sympos., Durham, 1977), vol.~36, 1979,
  pp.~137--203.

\bibitem{Thurstongeometry}
William Thurston, \emph{The geometry and topology of 3-manifolds}, Lecture
  notes at Princeton University, 1980.

\bibitem{Waldhausenirreducible}
Friedhelm Waldhausen, \emph{On irreducible {$3$}-manifolds which are
  sufficiently large}, Ann. of Math. (2) \textbf{87} (1968), 56--88.
  \MR{MR0224099 (36 \#7146)}

\bibitem{weidmann2024foldings}
Richard Weidmann and Thomas Weller, \emph{Foldings in relatively hyperbolic
  groups}, arXiv preprint arXiv:2403.17686 (2024).

\end{thebibliography}

\end{document}